\documentclass[oneside,english]{amsart}
\usepackage[T1]{fontenc}
\usepackage[latin9]{inputenc}
\usepackage{amsthm}
\usepackage{amssymb}
\usepackage{esint}

\makeatletter
\numberwithin{equation}{section}
\numberwithin{figure}{section}
\theoremstyle{plain}
\newtheorem{thm}{\protect\theoremname}
  \theoremstyle{plain}
  \newtheorem*{thm*}{\protect\theoremname}
  \theoremstyle{plain}
  \newtheorem{prop}[thm]{\protect\propositionname}
  \theoremstyle{plain}
  \newtheorem{lem}[thm]{\protect\lemmaname}
  \theoremstyle{plain}
  \newtheorem{cor}[thm]{\protect\corollaryname}
  \theoremstyle{remark}
  \newtheorem*{rem*}{\protect\remarkname}

\makeatother

\usepackage{babel}
  \providecommand{\corollaryname}{Corollary}
  \providecommand{\lemmaname}{Lemma}
  \providecommand{\propositionname}{Proposition}
  \providecommand{\remarkname}{Remark}
  \providecommand{\theoremname}{Theorem}
\providecommand{\theoremname}{Theorem}

\begin{document}

\title{$q$-binomials and non-continuity of the $p$-adic Fourier transform}

\author{Amit Ophir}

\author{Ehud de Shalit}

\date{21/6/16}
\begin{abstract}
Let $F$ be a finite extension of $\mathbb{Q}_{p}.$ We show that
every Schwartz function on $F$, with values in $\overline{\mathbb{Q}}_{p}$,
is the $p$-adic uniform limit of a sequence of Schwartz functions,
whose Fourier transforms tend uniformly to $0.$ The proof uses the
notion of $q$-binomial coefficients and some classical identities
between these polynomials.
\end{abstract}

\maketitle

\section{Introduction}

Let $p$ be a prime, and let $F$ be a finite extension of $\mathbb{Q}_{p}$
(the field of $p$-adic numbers). Let $C$ be an algebraically closed
field of characteristic 0. The Fourier transform $\mathcal{F}$ (defined
in section \ref{Fourier}) maps the space $\mathcal{S}$ of locally
constant, compactly supported, $C$-valued functions on $F,$ to itself.
When $C$ is the field of complex numbers, this is the basis for harmonic
analysis on $F,$ and it is well-known that $\mathcal{F}$ extends
to a unitary operator on $L^{2}(F,dx)$ where $dx$ is the normalized
Haar measure.

When $C$ is an algebraic closure of $\mathbb{Q}_{p},$ equipped with
the $p$-adic norm $|\cdot|_{p}$ (normalized in this paper so that
$|p|_{p}=p^{-1}$), no such $L^{2}$-theory is available. One may
ask how badly $\mathcal{F}$ behaves when we equip $\mathcal{S}$
with the simplest available norm, the sup norm. Our main result, which
turns out to be trickier than anticipated, is the following.
\begin{thm}
\label{Main}Let $\phi\in\mathcal{S}.$ Then there are $\phi_{n}\in\mathcal{S}$
which tend uniformly to $\phi,$ such that $\mathcal{F}(\phi_{n})$
tend uniformly to $0.$
\end{thm}
A few remarks are in order.

\medskip{}

(i) By the Stone-von Neumann theorem (see the discussion in section
\ref{sub:reduction}), our theorem is valid for all $\phi$ if and
only if it is valid for one non-zero $\phi$.

(ii) The completion of $\mathcal{S}$ in the sup norm is the space
$C_{0}(F)$ of continuous functions on $F$ ``vanishing at infinity''.
Consider the graph $\Gamma_{\mathcal{F}}$ of $\mathcal{F},$ as a
subspace of $\mathcal{S}\times\mathcal{S}\subset C_{0}(F)\times C_{0}(F).$
Equip the space $C_{0}(F)\times C_{0}(F)$ with the norm
\[
||(f_{1},f_{2})||=\max\left\{ ||f_{1}||_{sup},||f_{2}||_{sup}\right\} .
\]
An equivalent statement is that $\Gamma_{\mathcal{F}}$ is dense in
$C_{0}(F)\times C_{0}(F)$.

(iii) Let $\Lambda\subset\mathcal{S}$ be the unit ball for the sup
norm, and $\widehat{\Lambda}$ its image under the Fourier transform.
Both are integral structures for the action of the Heisenberg group
$H_{\psi}(F)$ (see section \ref{sub:Heisenberg}). Another formulation
of Theorem \ref{Main} is that
\[
\Lambda+\widehat{\Lambda}=\mathcal{S}.
\]

(iv) Let $\pi$ be a uniformizer of $F$, let $r\in\mathbb{Z},$ and
consider the subspaces $\mathcal{S}_{r}\subset\mathcal{S}$ and $C_{0}(F)_{r}\subset C_{0}(F)$
of functions which are constant on cosets of the additive subgroup
$\pi^{r}\mathcal{O}_{F}.$ While the Fourier transform is not a bounded
linear operator on the whole of $\mathcal{S},$ it \emph{is} bounded
when restricted to $\mathcal{S}_{r}$, and maps it to the subspace
$\mathcal{S}^{r}\subset\mathcal{S}$ of functions supported on $\pi^{-r}\mathcal{O}_{F}.$
It therefore extends by continuity to a map of $C_{0}(F)_{r}$ into
$C_{0}(F)^{r}$ (similarly defined). In particular, when $r=0$, the
Fourier transform of $f\in C_{0}(F)_{0}=C_{0}(F/\mathcal{O}_{F})$
is the ``Fourier series''
\[
\widehat{f}(x)=\sum_{y\in F/\mathcal{O}_{F}}f(y)\psi(xy)
\]
($x\in\mathcal{O}_{F}).$ In this context, Fresnel and de Mathan \cite{key-2,key-3},
and, independently, Amice and Escassut \cite{key-1} showed already
in 1973 that the map $f\mapsto\widehat{f}$ is non-injective. In other
words, they have constructed an $f\in C_{0}(F)_{0}$ and a sequence
$\phi_{n}\in\mathcal{S}_{0}$ converging uniformly to $f,$ for which
$\mathcal{F}(\phi_{n})$ converged uniformly to $0.$ Our theorem
is different in the sense that it demands the limit function $f$
to be a Schwartz function $\phi$. It is easy to see that in this
context, the approximating sequence $\phi_{n}$ does \emph{not} belong
to any fixed subspace $\mathcal{S}_{r}$ (or else, by continuity,
we would get the contradiction $\mathcal{F}(\phi)=0$).

(v) There does not seem to be any direct way to adapt the methods
of Fresnel and de Mathan (or Amice and Escassut) to prove our theorem.
If we could prove a Banach-space version of the Stone-von Neumann
theorem, on the topological irreducibility of $C_{0}(F)$ as a $p$-adic
Banach representation of the Heisenberg group $H(F),$ we would be
able to deduce our theorem from the example of Fresnel and de Mathan.
This remains, at present, out of reach, so our theorem can only be
regarded as an indirect evidence for such a theorem.

(vi) When $C$ is the field of complex numbers, and $\mathcal{S}$
is endowed with the sup norm rather than the $L^{2}$-norm, the Fourier
transform is again not continuous. Strangely, the analoguous theorem
is false then, see section \ref{sec:Complex}, where we also discuss
$l$-adic coefficients for $l\neq p.$

(vii) We do not treat in the present paper a local field $F$ of characteristic
$p$.

\medskip{}

The proof proceeds along some unexpected lines. We reduce easily to
the case $F=\mathbb{Q}_{p},$ where, as noted above, it suffices to
prove the theorem for $\phi_{0}=1_{\mathbb{Z}_{p}},$ the characteristic
function of $\mathbb{Z}_{p}.$ Fix an integer $r\ge0$ and consider
the subspace $\mathcal{S}_{r}^{r}$ of functions in $\mathcal{S}$
which are supported on $p^{-r}\mathbb{Z}_{p}$ and are constant modulo
$p^{r}\mathbb{Z}_{p}.$ This is a finite dimensional space, of dimension
$p^{2r},$ and $\mathcal{F}$ preserves it. Let
\[
\gamma_{r}=\inf_{\phi\in\mathcal{S}_{r}^{r}}\{\max\{||\phi-\phi_{0}||_{sup},||\mathcal{F}(\phi)||_{sup}\}\}.
\]
Our goal is to show that $\gamma_{r}\rightarrow0.$

For that purpose we compute the matrix of $\mathcal{F}$ in a suitable
basis of $\mathcal{S}_{r}^{r},$ where it comes out to be a multiple
of a Vandermonde matrix $Z$. The key tool in the proof is the use
of $q$-binomial coefficients and the $q$-binomial theorem, reviewed
in section \ref{sub:q-binomial}. These allow us to decompose
\[
Z=\,^{t}UDU
\]
as a product of a lower-triangular unipotent matrix $^{t}U,$ a diagonal
matrix $D$, and an upper-triangular unipotent matrix $U.$ While
this decomposition must be known, and is closely related to the $q$-Vandermonde
identity, we have not found it in the literature in the shape needed
here, so we gave a self-contained account of it. The matrix $U$ has
coefficients in the ring $\mathbb{Z}[\zeta]$ where $\zeta$ is a
primitive $p^{2r}$ root of unity in $C.$ The decomposition of $Z$
allows us to give a criterion for $\gamma_{r}$ to be less than a
fixed positive $\epsilon$. If this is the case, we say that $\epsilon$
is \emph{attainable }in $\mathcal{S}_{r}^{r}.$ The criterion is expressed
in terms of inequalities for certain quantities $\Omega_{r,n}$ for
$0\le n\le p^{2r}-1$. There are ``dual'' equivalent inequalities
for quantities that we denote $\omega_{r,n}.$ Taking both types of
inequalities into account, we show that for any given $\epsilon,$
our criterion holds if $r$ is large enough, as long as $n$ is outside
the ``critical interval'' $[\frac{1-\delta}{2}(p^{2r}-1),\frac{1+\delta}{2}(p^{2r}-1$){]}.
Here $\delta$ is a small pre-fixed positive number. To show that
the inequalities hold, for sufficiently large $r,$ even if $n$ belongs
to the critical interval, we study certain Laurent polynomials $\Omega_{r,n}(q)\in\mathbb{Z}[q,q^{-1}]$
for which $\Omega_{r,n}=\Omega_{r,n}(\zeta).$ Studying the divisibilty
of the $\Omega_{r,n}(q)$ by the cyclotomic polynomials $\Phi_{p^{j}}(q)$
for various values of $j$, allows us to obtain the required estimates
on $\Omega_{r,n}$ for $n$ in the critical interval.

\medskip{}

Searching the arXive, we found that $q$-analogues of classical polynomials
and congruences between them have recently attracted the attention
of a group of authors in China. This interest arose in connection
with combinatorial identities, hypergeometric series, and conjectures
of Rodriguez-Villegas. See \cite{key-5,key-6} and the references
therein. In fact, our Lemma \ref{Lemma 11} is taken from \cite{key-5}.
It seems to us that there might be something more fundamental lurking
behind this theory, and that its successful application to our problem
is not a coincidence.

\section{The Fourier transform}

\subsection{Definition\label{Fourier}}

Let $F$ be a local field of residue characteristic $p.$ Later on,
$F$ itself will be assumed to be of characteristic 0, i.e. a finite
extension of $\mathbb{Q}_{p},$ but for the meanwhile we allow it
to be of characteristic $p$ as well. Let $C$ be an algebraically
closed field of characteristic 0 (the \emph{field of coefficients})
and consider the Schwartz space $\mathcal{S}=C_{c}^{\infty}(F)$ of
locally constant, compactly supported, $C$-valued functions on $F.$

Let $\mathcal{O}_{F}$ be the ring of integers of $F$ and fix, once
and for all, an additive character
\[
\psi:F\rightarrow C^{\times},
\]
such that $\mathcal{O}_{F}$ is equal to its own annihilator under
the pairing $\left\langle x,y\right\rangle =\psi(xy).$ Let $dx$
be the unique $C$-valued distribution on the boolean algebra of open
and compact sets, which assigns to $\mathcal{O}_{F}$ the measure
1. The Fourier transform of $\phi\in\mathcal{S}$ is the function
$\mathcal{F}(\phi)=\widehat{\phi}\in\mathcal{S}$ given by
\[
\widehat{\phi}(x)=\int_{F}\phi(y)\psi(xy)dy.
\]
The Fourier inversion formula says that
\[
\mathcal{F}(\mathcal{F}(\phi))=\dot{\phi}
\]
where $\dot{\phi}(x)=\phi(-x).$ Thus $\mathcal{F}$ is ``almost
an involution''.

\subsection{Mesh and support}

Let $\pi$ be a uniformizer of $F.$ To every $\phi\in\mathcal{S}$
we attach two integers. Its \emph{support} $N(\phi)$ is
\[
N(\phi)=\min\left\{ n\in\mathbb{Z}|\,\phi(x)=0\,\forall x\notin\pi^{-n}\mathcal{O}_{F}\right\} ,
\]
and its \emph{mesh $M(\phi)$ }is
\[
M(\phi)=\min\left\{ n\in\mathbb{Z}|\,\phi(x+y)=\phi(x)\,\forall y\in\pi^{n}\mathcal{O}_{F}\right\} .
\]

The Fourier transform exchanges the support and the mesh, i.e.
\[
N(\mathcal{F}(\phi))=M(\phi),\,\,\,M(\mathcal{F}(\phi))=N(\phi).
\]

\subsection{The Heisenberg group and the Stone-von Neumann theorem\label{sub:Heisenberg}}

The \emph{Heisenberg group} $H(F)$ is the group of upper-triangular
unipotent matrices in $SL_{3}(F).$ We write
\[
(t;b,b')=\left(\begin{array}{ccc}
1 & b' & t\\
 & 1 & b\\
 &  & 1
\end{array}\right).
\]
Its center, which is also equal to the derived group $H(F)'$, is
the subgroup of all elements of the form $(t;0,0),$ and is isomorphic
to $F.$ We therefore have a short exact sequence of groups
\[
1\rightarrow F\rightarrow H(F)\rightarrow F\times F\rightarrow1,
\]
the projection to $F\times F$ given by $(t;b,b')\mapsto(b,b').$

Let $C^{1}$ be the unit circle $\{\lambda||\lambda|=1\}\subset C^{\times}.$
The push-out $H_{\psi}(F)$ of $H(F)$ by the character $\psi:F\rightarrow C^{1}$
is the group of all triples $[\lambda;b,b']$ with $\lambda\in C^{1}$
and $b,b'\in F.$ The group operation is given by the formula
\[
[\lambda;b,b'][\lambda_{1};b_{1},b'_{1}]=[\lambda\lambda_{1}\psi(b'b_{1});b+b_{1},b'+b'_{1}].
\]

The group $H(F)$ acts on $\mathcal{S}$ via
\[
((t;b,b')\phi)(x)=\psi(t)\psi(bx)\phi(x+b').
\]
This action factors through $H_{\psi}(F),$ and in the induced action
of the latter $\lambda\in C^{1}$ acts via multiplication by $\lambda.$

Let $V$ be a vector space over $C$ and $G$ a topological group.
Recall that a representation $\rho:G\rightarrow GL(V)$ is called
\emph{smooth }if the stabilizer of every vector is open. The celebrated
Stone-von Neumann theorem is the following \cite{key-4}:
\begin{thm*}
The space $\mathcal{S}$ with the action of $H(F)$ given above is
smooth and irreducible. Moreover, up to isomorphism it is the \emph{unique}
smooth irreducible representation of $H(F)$ in which the center acts
via the character $\psi.$
\end{thm*}
The relation between the action of $H(F)$ and $\mathcal{F}$ is the
following:
\[
\mathcal{F}\circ(t;b,b')=(t-bb';-b',b)\circ\mathcal{F}.
\]

\section{Complex and $l$-adic coefficients\label{sec:Complex}}

\subsection{Complex coefficients}

Let $C$ be the field of complex numbers $\mathbb{C}.$ The Haar distribution
$dx$ is, in this case, a Haar measure. The Fourier transform preserves
the $L^{2}$ metric on $\mathcal{S}$ and extends to a unitary operator
on $L^{2}(F,dx).$ In contrast, if we endow $\mathcal{S}$ with the
sup norm $||\phi||=\sup_{x\in F}|\phi(x)|$, $\mathcal{F}$ is unbounded,
so does not extend by continuity to the Banach space $C_{0}(F),$
the completion of $\mathcal{S}.$ It is natural to consider its graph
\[
\Gamma_{\mathcal{F}}=\{(\phi,\widehat{\phi})|\,\phi\in\mathcal{S}\}\subset\mathcal{S}\times\mathcal{S}\subset C_{0}(F)\times C_{0}(F)
\]
and the closure $C_{\mathcal{F}}(F)$ of $\Gamma_{\mathcal{F}}$ in
$C_{0}(F)\times C_{0}(F)$ under the norm
\[
||(f_{1},f_{2})||=\max\{||f_{1}||,||f_{2}||\}.
\]
This is a linear subspace of $C_{0}(F)\times C_{0}(F),$ and if the
projection $pr_{1}:C_{\mathcal{F}}(F)\rightarrow C_{0}(F)$ to the
first factor is \emph{injective, }we may extend $\mathcal{F}$ to
$\mbox{Im}(pr_{1})$, setting
\[
\mathcal{F}=pr_{2}\circ pr_{1}^{-1}.
\]
In this way the image of $pr_{1}$ becomes a natural subspace of $C_{0}(F)$
to which we can extend the Fourier transform. All hinges on $pr_{1}$
being injective. We are indebted to Izzy Katsenelson for pointing
out to us that this is indeed the case.
\begin{prop}
The projection $pr_{1}:C_{\mathcal{F}}(F)\rightarrow C_{0}(F)$ is
injective.\end{prop}
\begin{proof}
Recall that if $h\in\mathcal{S}$ and $g\in C_{0}(F)$ their convolution
is defined by
\[
h*g(x)=\int_{F}h(y)g(x-y)dy.
\]
If also $g\in\mathcal{S}$ then $h*g\in\mathcal{S}$ and
\[
\widehat{h*g}=\widehat{h}\widehat{g},\,\,\,\widehat{hg}=\widehat{h}*\widehat{g}.
\]

Now let $\Delta,\phi\in\mathcal{S}$ and $f\in C_{0}(F).$ We claim
that
\[
\mbox{(i)\,\,}||\Delta*f||_{\infty}\le||\Delta||_{L^{1}}||f||_{\infty}\mbox{\,\,\,(ii)\,\,\ensuremath{||\Delta*\phi||_{\infty}\le||\widehat{\Delta}||_{L^{1}}||\widehat{\phi}||_{\infty}.}}
\]
Point (i) is clear. For (ii) note that
\[
||\widehat{\Delta}*\widehat{\phi}||_{\infty}=||\widehat{\Delta\phi}||_{\infty}\le||\Delta\phi||_{L^{1}}\le||\Delta||_{L^{1}}||\phi||_{\infty},
\]
substitute $\widehat{\Delta}$ for $\Delta$ and $\widehat{\phi}$
for $\phi$, and use the Fourier inversion formula.

We can now prove the proposition. Suppose that $(f,0)\in C_{\mathcal{F}}(F)$
with $f\neq0$ and pick a sequence $\phi_{n}\in\mathcal{S}$ such
that $\phi_{n}\rightarrow f$ but $\widehat{\phi_{n}}\rightarrow0$
uniformly. Take any $\Delta\in\mathcal{S}$ such that $\Delta*f\neq0.$
Then
\[
\Delta*f=\Delta*(f-\phi_{n})+\Delta*\phi_{n}.
\]
By (i) $\Delta*(f-\phi_{n})\rightarrow0$ uniformly and by (ii) $\Delta*\phi_{n}\rightarrow0$
uniformly. This is a contradiction.
\end{proof}

\subsection{$l$-adic coefficients}

Let $l\neq p$ be a prime and take for $C$ an algebraic closure of
$\mathbb{Q}_{l}$, endowed with the $l$-adic metric. Let $||\phi||$
be the sup norm on $\mathcal{S}$ and as usual, $C_{0}(F)$ its $(l$-adic)
Banach completion. In this case the Fourier transform is unitary,
\[
||\widehat{\phi}||=||\phi||
\]
for every $\phi\in\mathcal{S}.$ Indeed, the ultra-metric inequality
and the fact that $p$ is an $l$-adic unit, imply $||\widehat{\phi}||\le||\phi||,$
and by Fourier inversion we get an equality. If $||\phi-1_{\mathcal{O}_{F}}||<1$
it follows that $||\widehat{\phi}-1_{\mathcal{O}_{F}}||<1$, so $||\widehat{\phi}||=1.$
The analogue of Theorem \ref{Main} is blatantly flase. To the contrary,
$\mathcal{F}$ extends by continuity to a unitary isomorphism of $C_{0}(F)$
onto itself.

\section{$p$-adic coefficients and the Main Theorem}

\subsection{An easy reduction\label{sub:reduction}}

From now on $F$ will be assumed to be a finite extension of $\mathbb{Q}_{p},$
and $C$ an algebraic closure of $\mathbb{Q}_{p}$. We normalize the
$p$-adic absolute value on $C$ by $|p|=p^{-1}$ and let again $||\phi||$
denote the sup norm of $\phi$. The remainder of our work will be
devoted to a proof of Theorem \ref{Main}.

It is easy to check that the space of functions $\phi\in\mathcal{S}$
for which there exists a sequence $\phi_{n}$ as in Theorem \ref{Main}
is closed under translation and also under multiplication by $\psi(bx)$
for any $b\in F.$ It follows from the Stone-von Neumann theorem that
this space is either $0$ or the whole of $\mathcal{S}.$ It is thus
enough to prove the theorem for one specific choice $\phi=\phi_{0}\neq0$,
for example $\phi_{0}=1_{\mathcal{O}_{F}}.$

Next, view $F$ as a vector space of dimension $d=[F:\mathbb{Q}_{p}]$
over $\mathbb{Q}_{p}$ and fix a basis $\alpha_{1},\ldots,\alpha_{d}$.
Then we obtain an isomorphism
\[
C_{c}^{\infty}(F)\simeq C_{c}^{\infty}(\mathbb{Q}_{p})\otimes\cdots\otimes C_{c}^{\infty}(\mathbb{Q}_{p})
\]
and it becomes clear that it is enough to prove the theorem for $F=\mathbb{Q}_{p}.$
For example, if $d=2,$ assume that $\alpha_{1},\alpha_{2}$ is also
a basis of $\mathcal{O}_{F}$ over $\mathbb{Z}_{p}$. The characteristic
function $\Phi$ of $\mathcal{O}_{F}$ is then
\[
\Phi(x_{1}\alpha_{1}+x_{2}\alpha_{2})=\phi(x_{1})\phi(x_{2})
\]
where $\phi$ is the characteristic function of $\mathbb{Z}_{p}.$
For the additive character of $F$ we take $\Psi=\psi\circ Tr_{F/\mathbb{Q}_{p}}$where
$\psi$ is an additive character of $\mathbb{Q}_{p}.$ This may not
make $\mathcal{O}_{F}$ self-dual, but the validity of the theorem
will not be affected by the resulting modification in the Fourier
transform. Let $\beta_{1},\beta_{2}$ be the basis dual to $\alpha_{1},\alpha_{2}$
under the trace pairing. Assume that $\phi_{n}\rightarrow\phi$ but
$\widehat{\phi_{n}}\rightarrow0$ and let $\Phi_{n}(x_{1}\alpha_{1}+x_{2}\alpha_{2})=\phi_{n}(x_{1})\phi_{n}(x_{2})$.
Then $\Phi_{n}\rightarrow\Phi$ and
\[
\widehat{\Phi_{n}}(x_{1}\beta_{1}+x_{2}\beta_{2})=\int\int\phi_{n}(y_{1})\phi_{n}(y_{2})\Psi((x_{1}\beta_{1}+x_{2}\beta_{2})(y_{1}\alpha_{1}+y_{2}\alpha_{2}))dy_{1}dy_{2}
\]
\[
=\widehat{\phi_{n}}(x_{1})\widehat{\phi_{n}}(x_{2})\rightarrow0.
\]

We therefore assume that $F=\mathbb{Q}_{p}$ and $\phi_{0}=1_{\mathbb{Z}_{p}}$
is the characteristic function of $\mathbb{Z}_{p}$. For every integer
$r\ge0$ let $\mathcal{S}_{r}^{r}$ be the vector space of $C$-valued
functions supported on $p^{-r}\mathbb{Z}_{p}$ and constant modulo
$p^{r}\mathbb{Z}_{p}$. It is invariant under $\mathcal{F}$, of dimension
$p^{2r}$, and admits the basis
\[
v_{i}=1_{\frac{i}{p^{r}}+p^{r}\mathbb{Z}_{p}}\,\,\,(0\le i\le p^{2r}-1).
\]
Clearly
\[
\phi_{0}=\sum_{j=0}^{p^{r}-1}v_{jp^{r}}.
\]

We introduce the positive quantity
\[
\gamma_{r}=\inf_{\phi\in\mathcal{S}_{r}^{r}}\left\{ \max\{||\phi-\phi_{0}||,||\widehat{\phi}||\right\} .
\]
The sequence $\gamma_{r}$ is $\le1$ and decreasing, and we say that
$\epsilon>0$ is \emph{attainable} in $\mathcal{S}_{r}^{r}$ if $\gamma_{r}\le\epsilon.$
Theorem $\ref{Main}$ is equivalent to the statement
\[
\mbox{Every \ensuremath{\epsilon>0} is attainable in \ensuremath{\mathcal{S}_{r}^{r}} if \ensuremath{r} is large enough.}
\]

Fix $\psi$ and let $\zeta=\psi(1/p^{2r}),$ a primitive $p^{2r}$
root of unity. The Fourier transform of $v_{j}$ is
\[
\widehat{v_{j}}(x)=\int_{p^{r}\mathbb{Z}_{p}}\psi(x(\frac{j}{p^{r}}+y))dy=\frac{1}{p^{r}}\sum_{i=0}^{p^{2r}-1}\zeta^{ij}v_{i}(x).
\]
It follows that in the basis $(v_{i})$ the matrix of $\mathcal{F}$
is $p^{-r}Z,$ where $Z$ is the Vandermonde matrix
\[
Z=(\zeta^{ij})_{0\le i,j\le p^{2r}-1}.
\]

\subsection{$q$-binomial coefficients\label{sub:q-binomial}}

We review some well-known identities, working in the ring $\mathbb{Z}[q,q^{-1}]$
of Laurent polynomials. Substituting $q=1$ yields the classical identities
between binomail coefficients, but we shall be concerned later with
$q=\zeta,$ a primitive $p^{2r}$ root of unity.

For $n\ge0$ let
\[
[n]_{q}=\frac{1-q^{n}}{1-q}=1+q+\cdots+q^{n-1}
\]
\[
[n]_{q}!=[1]_{q}\cdots[n]_{q}
\]
$([0]_{q}!=1)$ and define the \emph{Pochhammer symbol}
\[
(X;q)_{n}=\prod_{k=0}^{n-1}(1-Xq^{k})\in\mathbb{Z}[q][X]
\]
($(X;q)_{0}=1).$ Substituting $X=q$ yields
\[
(q;q)_{n}=(1-q)^{n}[n]_{q}!=\prod_{k=1}^{n}(1-q^{k}).
\]

The $q$\emph{-binomial coefficients} are
\[
{n \brack k}_{q}=\frac{[n]_{q}!}{[k]_{q}![n-k]_{q}!}=\frac{(q;q)_{n}}{(q;q)_{k}(q;q)_{n-k}}.
\]
These are in fact \emph{polynomials} in $\mathbb{\mathbb{Z}}[q]$
of degree $k(n-k).$ One way to see it is via the $q$\emph{-Pascal
identity:
\[
{n \brack k}_{q}+q^{n+1-k}{n \brack k-1}_{q}={n+1 \brack k}_{q}.
\]
}When $k$ is an integer outside the interval $[0,n]$ we put ${n \brack k}_{q}=0.$
\begin{lem}
($q$-binomial theorem) We have the identity
\[
(Xq;q)_{n}=\prod_{k=1}^{n}(1-Xq^{k})=\sum_{k=0}^{n}\left((-1)^{k}q^{{k+1 \choose 2}}{n \brack k}_{q}\right)X^{k}.
\]

\end{lem}
In the proof of Theorem \ref{Main}, the Laurent polynomials
\begin{equation}
\Omega_{r,n}(q)=\sum_{k=0}^{\lfloor n/p^{r}\rfloor}(-1)^{kp^{r}}q^{{kp^{r}+1 \choose 2}-nkp^{r}}{n \brack kp^{r}}_{q}\in\mathbb{Z}[q,q^{-1}]\label{eq:Omega(q)}
\end{equation}
($0\le r$, $0\le n\le p^{2r}-1),$ will play an important role.
\begin{lem}
We have the formula
\begin{equation}
\Omega_{r,n}(q)=\frac{1}{p^{r}}\sum_{\eta^{p^{r}}=1}\prod_{l=0}^{n-1}(1-q^{-l}\eta).\label{eq:q-expression}
\end{equation}
\end{lem}
\begin{proof}
We may write
\[
\Omega_{r,n}(q)=\frac{1}{p^{r}}\sum_{\eta^{p^{r}}=1}\sum_{k=0}^{n}(-1)^{k}q^{{k+1 \choose 2}-nk}{n \brack k}_{q}\eta^{k}.
\]
Using the $q$-binomial theorem, substituting $q^{-n}\eta$ for $X$,
we get the expression
\[
\Omega_{r,n}(q)=\frac{1}{p^{r}}\sum_{\eta^{p^{r}}=1}\prod_{k=1}^{n}(1-q^{k-n}\eta)=\frac{1}{p^{r}}\sum_{\eta^{p^{r}}=1}\prod_{l=0}^{n-1}(1-q^{-l}\eta).
\]

\end{proof}
Observe that it is not clear from the last formula that $\Omega_{r,n}(q)$
has integral coefficients. Nevertheless, it will allow us to prove
that $\Omega_{r,n}$ is divisible by certain cyclotomic polynomials
over $\mathbb{Q}$, and by Gauss' lemma this divisibilty will hold
also over $\mathbb{Z}.$

\subsection{A decomposition of the Vandermonde}

Let $d\ge1$ be an integer and consider the matrix
\[
Z=Z_{q,d}=(q^{ij})_{0\le i,j\le d-1}.
\]
Let $U=U_{q,d}=(u_{n,k})$ be the upper-triangular unipotent matrix
with entries in $\mathbb{Z}[q]$ defined by
\[
u_{n,k}={k \brack n}_{q}
\]
for $0\le n\le k\le d-1.$ Similarly let $V=V_{q,d}=(v_{n,k})$ be
the lower-triangular unipotent matrix with entries in $\mathbb{Z}[q]$
defined by
\[
v_{n,k}=(-1)^{n-k}q^{{n-k \choose 2}}{n \brack k}_{q}
\]
for $0\le k\le n\le d-1$. Let $D=D_{q,d}=(w_{i,j})$ be the diagonal
matrix with
\[
w_{n,n}=(-1)^{n}q^{{n \choose 2}}(q;q)_{n}
\]
for $0\le n\le d-1.$
\begin{prop}
(i) One has $V_{q,d}=\,^{t}U_{q,d}^{-1}.$ Equivalently,
\[
\sum_{k=0}^{n}(-1)^{n-k}q^{{n-k \choose 2}}{n \brack k}_{q}{k \brack m}_{q}=\delta_{n,m}.
\]
(ii) One has the decomposition
\[
Z_{q,d}=\,^{t}U_{q,d}D_{q,d}U_{q,d}.
\]
\end{prop}
\begin{cor}
\label{cor:U=00003DLZ}Let $L_{q,d}=D_{q,d}^{-1}V_{q,d}.$ Then $U_{q,d}=L_{q,d}Z_{q,d}.$\end{cor}
\begin{proof}
Define $L_{q,d}$ as in the corollary. Thus, noting that ${n-l \choose 2}-{n \choose 2}={l+1 \choose 2}-nl,$
\[
(L_{q,d})_{n,l}=(-1)^{l}\frac{q^{{l+1 \choose 2}-nl}}{(q;q)_{n}}{n \brack l}_{q}
\]
and
\[
\sum_{l=0}^{d-1}(L_{q,d})_{n,l}(Z_{q,d})_{l,k}=\sum_{l=0}^{d-1}(-1)^{l}\frac{q^{{l+1 \choose 2}-nl}}{(q;q)_{n}}{n \brack l}_{q}q^{lk}
\]
\[
=\frac{1}{(q;q)_{n}}\sum_{l=0}^{n}(-1)^{l}q^{{l+1 \choose 2}}{n \brack l}_{q}q^{(k-n)l}
\]

\[
=\frac{1}{(q;q)_{n}}\prod_{l=1}^{n}(1-q^{k-n+l})=\frac{(q;q)_{k}}{(q;q)_{n}(q;q)_{k-n}}={k \brack n}_{q}=(U_{q,d})_{n,k}.
\]
We have used the $q$-binomial theorem in the passage from the second
to third line. Note also that if $k<n$ the product over $(1-q^{k-n+l})$
vanishes because of the term with $l=n-k$.

This verifies the corollary, from which we deduce that
\[
Z_{q,d}=V_{q,d}^{-1}D_{q,d}U_{q,d}.
\]
Since $Z_{q,d}$ is a symmetric matrix, we obtain also $V_{q,d}=\,^{t}U_{q,d}^{-1},$
because the decomposition into a product of lower-triangular unipotent,
diagonal and upper-triangular unipotent must be unique.
\end{proof}

\subsection{A criterion for attainability}

Write $d=p^{2r}.$ If we identify $\mathcal{S}_{r}^{r}$ with the
space $C^{d}$ via the basis $(v_{i})_{0\le i\le d-1}$, then the
sup norm on $\mathcal{S}_{r}^{r}$ corresponds to the sup norm on
$C^{d}$. The function $\phi_{0}=1_{\mathbb{Z}_{p}}$ corresponds
to the column
\[
J=\sum_{j=0}^{p^{r}-1}v_{jp^{r}}=\,^{t}(1,0,\ldots,0;1,0,\ldots,0;\ldots\ldots;1,0,\ldots,0).
\]
Thus $\epsilon$ is attainable in $\mathcal{S}_{r}^{r}$ if and only
if there are $x,y\in C^{d}$ with $||x||\le\epsilon,\,\,||y||\le\epsilon$
and
\[
\frac{1}{p^{r}}Z_{\zeta,d}x=y+J.
\]
Equivalently,
\[
U_{\zeta,d}x=p^{r}L_{\zeta,d}\left(y+J\right).
\]

Note that $U_{\zeta,d}$ is an isometry of $C^{d},$ i.e. $||U_{\zeta,d}x||=||x||.$
Indeed, as all its entries are in $\mathcal{O}_{C}$ it is contracting,
but the same is true of $U_{\zeta,d}^{-1}.$ We conclude that $\epsilon$
is attainable in $\mathcal{S}_{r}^{r}$ if and only if there exists
a $y\in C^{d}$ with $||y||\le\epsilon$ and
\begin{equation}
||L_{\zeta,d}\left(y+J\right)||\le p^{r}\epsilon.\label{eq:estimate}
\end{equation}
The advantage of the latter equation over the original one (with the
$Z_{\zeta,d}$) is that $L_{\zeta,d}$ is triangular, so the equation
renders itself to an inductive argument. The following lemma is our
key criterion.
\begin{lem}
For $0\le n\le d-1$ let
\begin{equation}
\Omega_{r,n}=\Omega_{r,n}(\zeta)=\sum_{k=0}^{\lfloor n/p^{r}\rfloor}(-1)^{kp^{r}}\zeta^{{kp^{r}+1 \choose 2}-nkp^{r}}{n \brack kp^{r}}_{\zeta}.\label{eq:Omega}
\end{equation}
Then $\epsilon$ is attainable in $\mathcal{S}_{r}^{r}$ if and only
if
\begin{equation}
|\Omega_{r,n}|\le\max\left\{ 1,p^{r}|(\zeta;\zeta)_{n}|\right\} \cdot\epsilon\label{eq:criterion}
\end{equation}
for all $0\le n\le d-1.$\end{lem}
\begin{proof}
We start by showing sufficiency. The $n$th entry of $L_{\zeta,d}\left(y+J\right)$
is
\[
L_{\zeta,d}\left(y+J\right)_{n}=\sum_{k=0}^{d-1}(-1)^{k}\frac{\zeta^{{k+1 \choose 2}-nk}}{(\zeta;\zeta)_{n}}{n \brack k}_{\zeta}(y_{k}+J_{k})
\]

\[
=\frac{1}{(\zeta;\zeta)_{n}}\left\{ \sum_{k=0}^{n}(-1)^{k}\zeta^{{k+1 \choose 2}-nk}{n \brack k}_{\zeta}y_{k}+\Omega_{r,n}\right\} .
\]
Suppose $y_{0},\ldots,y_{n-1}$ have been chosen so that $|y_{i}|\le\epsilon$
and also $|L_{\zeta,d}\left(y+J\right)_{i}|\le p^{r}\epsilon$ for
all $0\le i\le n-1$ (if $n=0$ this condition is empty). If $|\Omega_{r,n}|\le\epsilon,$
we can select $y_{n}$ so that $L_{\zeta,d}\left(y+J\right)_{n}=0$
and $|y_{n}|\le\epsilon.$ If $|\Omega_{r,n}|>\epsilon$, choose $y_{n}$
arbitrarily so that $|y_{n}|\le\epsilon.$ Then
\[
|L_{\zeta,d}\left(y+J\right)_{n}|=|\Omega_{r,n}/(\zeta;\zeta)_{n}|,
\]
so if (\ref{eq:criterion}) is satisfied, $|L_{\zeta,d}\left(y+J\right)_{n}|\le p^{r}\epsilon$
too. Proceeding by induction on $n$ we construct a vector $y$ such
that $||y||\le\epsilon$ and (\ref{eq:estimate}) holds.

Necessity is proved in the same way. Suppose that a $y$ as above
has been constructed. Then if $|\Omega_{r,n}|>\epsilon,$ we must
have $|\Omega_{r,n}|=|(\zeta;\zeta)_{n}||L_{\zeta,d}\left(y+J\right)_{n}|\le p^{r}|(\zeta;\zeta)_{n}|\epsilon.$
\end{proof}

\subsection{A dual criterion}

For $0\le n\le d-1$ let
\[
\omega_{r,n}=\zeta^{{n+1 \choose 2}}\frac{p^{r}}{(\zeta;\zeta)_{n}}\Omega_{r,n}.
\]

\begin{lem}
(i) We have
\begin{equation}
\omega_{r,n}=\zeta^{{n+1 \choose 2}}\sum_{k=0}^{p^{r}-1}{kp^{r} \brack n}_{\zeta}.\label{eq:omega}
\end{equation}
(ii) For any $r,n$ the criterion (\ref{eq:criterion}) is equivalent
to
\[
|\omega_{r,n}|\le\max\left\{ p^{-r}|(\zeta;\zeta)_{n}|^{-1},1\right\} \cdot\epsilon.
\]
\end{lem}
\begin{rem*}
Part (ii), the reformulation of the criterion, is of course obvious.
What is new is the expression given in part (i). Note that both $\Omega_{r,n}$
and $\omega_{r,n}$ lie in $\mathbb{Z}[\zeta].$ Note that in (\ref{eq:Omega})
the running index appears in the denominators of the $\zeta$-binomial
coefficients, while in (\ref{eq:omega}) it appears in the numerators.
Once we obtain an estimate on $|(\zeta;\zeta)_{n}|$, the expression
$(\ref{eq:Omega})$ will be useful for $n\in[0,\frac{1-\delta}{2}(d-1)],$
while (\ref{eq:omega}) will be useful for $n\in[\frac{1+\delta}{2}(d-1),d-1].$
Here $\delta>0$ is fixed. The smaller we make $\delta$ the larger
$r$ will have to be taken to guarantee the criterion for attainability
in these two regions. Treating $n$ in the ``critical interval''
$[\frac{1-\delta}{2}(d-1),\frac{1+\delta}{2}(d-1)]$ will take some
more effort.\end{rem*}
\begin{proof}
We only have to prove (i). Using (\ref{eq:q-expression}) we may write
\begin{equation}
\Omega_{r,n}=\frac{1}{p^{r}}\sum_{\eta^{p^{r}}=1}\prod_{l=0}^{n-1}(1-\zeta^{-l}\eta).\label{eq:expression}
\end{equation}
The $p^{r}$ roots of unity $\eta$ are just the $\zeta^{kp^{r}}$
for $0\le k\le p^{r}-1.$ We get
\[
\Omega_{r,n}=\frac{1}{p^{r}}\sum_{k=0}^{p^{r}-1}\prod_{l=0}^{n-1}(1-\zeta^{kp^{r}-l})=\frac{1}{p^{r}}\sum_{k=0}^{p^{r}-1}\frac{(\zeta;\zeta)_{kp^{r}}}{(\zeta;\zeta)_{kp^{r}-n}}=\frac{(\zeta;\zeta)_{n}}{p^{r}}\sum_{k=0}^{p^{r}-1}{kp^{r} \brack n}_{\zeta}.
\]
Note that the terms with $kp^{r}<n$ drop out from the sum.

An alternative proof is to observe that since $1_{\mathbb{Z}_{p}}$
is invariant under the Fourier transform, $Z_{\zeta,d}J=p^{r}J$,
hence from Corollary \ref{cor:U=00003DLZ}
\[
p^{r}L_{\zeta,d}J=L_{\zeta,d}Z_{\zeta,d}J=U_{\zeta,d}J.
\]
Noting that the $n$th entry on the left is $\zeta^{-{n+1 \choose 2}}\omega_{r,n}$,
we get (\ref{eq:omega}).
\end{proof}
The formulae for $\omega_{r,n}$ can be summarized in the following
expression for the generating function:

\[
\sum_{n=0}^{d-1}(-1)^{n}\omega_{r,n}X^{n}=\sum_{k=0}^{p^{r}-1}\prod_{l=1}^{kp^{r}}(1-\zeta^{l}X).
\]
Simply expand the products using the $q$-binomial theorem and collect
terms.

\subsection{An estimate for $|(\zeta;\zeta)_{n}|$}

We turn our attention to
\[
(\zeta;\zeta)_{n}=\prod_{k=1}^{n}(1-\zeta^{k}).
\]
Normalize the $p$-adic valuation $v_{p}:C^{\times}\rightarrow\mathbb{Q}$
so that $v_{p}(p)=1.$
\begin{lem}
Let $0\le n\le d-1$. Write
\[
n=\sum_{i=0}^{s}a_{i}p^{i},\,\,\,\,0\le a_{i}\le p-1
\]
($0\le s\le2r-1$) and
\[
v_{p}((\zeta;\zeta)_{n})=\frac{\beta_{p}(n)}{(p-1)p^{2r-1}}.
\]
Then
\[
\beta_{p}(n)=n+\frac{p-1}{p}\sum_{i=0}^{s}ia_{i}p^{i}.
\]
\end{lem}
\begin{proof}
Clearly
\[
\beta_{p}(n)=\sum_{k=0}^{\infty}p^{k}\left(\lfloor\frac{n}{p^{k}}\rfloor-\lfloor\frac{n}{p^{k+1}}\rfloor\right).
\]
By Abel's summation formula
\[
\beta_{p}(n)=n+\frac{p-1}{p}\sum_{k=1}^{\infty}p^{k}\lfloor\frac{n}{p^{k}}\rfloor=n+\frac{p-1}{p}\sum_{i=0}^{s}ia_{i}p^{i}.
\]
\end{proof}
\begin{cor}
\label{cor:beta}We have the estimates
\[
\frac{p-1}{p}n\lfloor\log_{p}n\rfloor+1\le\beta_{p}(n)\le\frac{p-1}{p}n\lfloor\log_{p}n\rfloor+n.
\]
\end{cor}
\begin{proof}
The corollary follows from the inequalities
\[
n\lfloor\log_{p}n\rfloor-\frac{p}{p-1}(n-1)\le n\lfloor\log_{p}n\rfloor-\sum_{k=1}^{\lfloor\log_{p}n\rfloor}p^{k}\le\sum_{k=1}^{\infty}p^{k}\lfloor\frac{n}{p^{k}}\rfloor\le n\lfloor\log_{p}n\rfloor.
\]

\end{proof}

\subsection{Verification of the criterion (\ref{eq:criterion}) outside the critical
interval}
\begin{lem}
Let $\epsilon>0$ and $0<\delta<1$. Then if $r$ is large enough
and $n\notin[\frac{1-\delta}{2}(d-1),\frac{1+\delta}{2}(d-1)]$ the
inequality (\ref{eq:criterion}) holds.\end{lem}
\begin{proof}
Let us treat first the range $n\le\frac{1-\delta}{2}(p^{2r}-1).$
Since $\log_{p}|\Omega_{r,n}|\le0,$ it is enough to show that for
$r$ large enough and all $n$ in this range
\[
0\le r+\log_{p}|(\zeta;\zeta)_{n}|+\log_{p}\epsilon.
\]
Now $\log_{p}|(\zeta;\zeta)_{n}|=-v_{p}((\zeta;\zeta)_{n})$ so it
is enough to show that
\[
v_{p}((\zeta;\zeta)_{n})\le r+\log_{p}\epsilon.
\]
Using the last corollary we have
\[
v_{p}((\zeta;\zeta)_{n})=\frac{\beta_{p}(n)}{(p-1)p^{2r-1}}\le\frac{n}{p^{2r}}\lfloor\log_{p}n\rfloor+\frac{n}{(p-1)p^{2r-1}}\le\frac{1-\delta}{2}\cdot2r+1=(1-\delta)r+1.
\]
Since $\delta>0$, for $r$ large enough we are done, no matter how
negative $\log_{p}\epsilon$ is.

For $n\ge\frac{1+\delta}{2}(p^{2r}-1)$ we use the dual criterion
in terms of (\ref{eq:omega}). Since $\log_{p}|\omega_{r,n}|\le0,$
it is enough to show that for $r$ large enough and all $n$ in this
range
\[
0\le-r-\log_{p}|(\zeta;\zeta)_{n}|+\log_{p}\epsilon.
\]
Equivalently, we have to show
\[
r-\log_{p}\epsilon\le v_{p}((\zeta;\zeta)_{n}).
\]
Now we use the lower bound from Corollary \ref{cor:beta}:

\[
v_{p}((\zeta;\zeta)_{n})=\frac{\beta_{p}(n)}{(p-1)p^{2r-1}}\ge\frac{n}{p^{2r}}\lfloor\log_{p}n\rfloor\ge\frac{1+\delta}{2}\cdot(1-\frac{1}{p^{2r}})\cdot(2r-1)\ge(1+\delta)r-1.
\]
Once again, since $\delta>0,$ for $r$ large enough we are done.
\end{proof}

\subsection{Conclusion of the proof of the main theorem}

For $n$ near the center of the interval $[0,d-1],$ the quantity
$p^{r}|(\zeta;\zeta)_{n}|,$ which appears (with its inverse) in both
versions of our key criterion, is roughly 1. It is therefore neither
very large, nor very small (when $r\rightarrow\infty)$ and the integrality
of $\Omega_{r,n}$ and $\omega_{r,n}$ alone does not suffice. In
order to estimate $\Omega_{r,n}$ in the range $n\in[\frac{1-\delta}{2}(d-1),\frac{1+\delta}{2}(d-1)]$,
we study, following Hao-Pan, the divisibility of the polynomials $\Omega_{r,n}(q)$
(\ref{eq:Omega(q)}) by certain cyclotomic polynomials.

Let $\Phi_{m}(q)$ be the $m$th cyclotomic polynomial, the product
of $(q-\alpha)$ for all the primitive $m$th roots of unity $\alpha.$
It is a monic irreducible polynomial in $\mathbb{Z}[q].$ The following
is Theorem 1.1 of \cite{key-5}.
\begin{lem}
\label{Lemma 11}In $\mathbb{Z}[q,q^{-1}]$ we have
\[
\prod_{j=r}^{2r-1}\Phi_{p^{j}}(q)^{\lfloor n/p^{j}\rfloor}|\Omega_{r,n}(q).
\]
\end{lem}
\begin{proof}
We use the expression derived in (\ref{eq:q-expression}). As the
$\Phi_{p^{j}}$ are irreducible, monic and distinct, it is enough
to check that $\Phi_{p^{j}}(q)^{\lfloor n/p^{j}\rfloor}|\Omega_{r,n}(q)$
over $\mathbb{Q}$. Let $\xi$ be a primitive root of unity of order
$p^{j}$ with $r\le j\le2r-1.$ Then any $p^{r}$ root of 1 $\eta$
is also a $p^{j}$ root of 1, hence $\eta=\xi^{l}$ for a unique $0\le l\le p^{j}-1.$
It follows that the multiplicity of $\xi$ in $\prod_{l=0}^{n-1}(1-q^{-l}\eta)$
is $\lfloor n/p^{j}\rfloor.$ Since this is true for every $\eta$
and every $\xi$ the lemma follows.
\end{proof}
Substitute $\zeta$ for $q.$ If $\xi$ is a $p^{j}$ root of unity
and $j\le2r-1,$ $v_{p}(\zeta-\xi)=1/(p-1)p^{2r-1},$ so
\[
v_{p}(\Phi_{p^{j}}(\zeta))=\frac{1}{p^{2r-j}}.
\]
It follows that
\begin{equation}
v_{p}(\Omega_{r,n})\ge\sum_{j=r}^{2r-1}\lfloor\frac{n}{p^{j}}\rfloor\cdot\frac{1}{p^{2r-j}}\ge\sum_{j=r}^{2r-1}(\frac{n}{p^{j}}-1)\cdot\frac{1}{p^{2r-j}}\ge\frac{rn}{p^{2r}}-\frac{1}{p-1}.\label{eq:last estimate}
\end{equation}

We can now conclude the proof of Theorem \ref{Main}. Fix $\delta>0$
and let $I_{r}$ be the interval $[\frac{1-\delta}{2}(d-1),\frac{1+\delta}{2}(d-1)].$
The last formula shows that
\[
\max_{n\in I_{r}}|\Omega_{r,n}|\rightarrow0
\]
as $r\rightarrow\infty,$ hence ($\ref{eq:criterion}$) holds in $I_{r}$
if $r$ is large enough. As it was previously shown to hold outside
$I_{r}$, the proof is concluded. Our estimate (\ref{eq:last estimate})
can be used also to verify the criterion for $n$ in the ``right
half'' $[\frac{1+\delta}{2}(d-1),d-1],$ where previously we have
used the ``dual criterion'' by means of $\omega_{r,n},$ but this
is not surprising, since both arguments are based on formula (\ref{eq:q-expression}).

\end{document}